\documentclass[12pt]{amsart}

\usepackage{amsmath}
\usepackage{amssymb}
\usepackage{amsthm}
\usepackage{cite}
\usepackage{caption}
\usepackage{tikz}
\usepackage{amsfonts}
\usepackage{verbatim}

\usetikzlibrary{shapes,snakes}

\makeatletter
\@namedef{subjclassname@2010}{%
  \textup{2010} Mathematics Subject Classification}
\makeatother

  \usepackage{ifthen}
 \newcommand{\mymarginpar}[1]{%
    \marginpar{\ifthenelse{\isodd{\arabic{page}}}{\flushleft 
#1}{\flushright #1}}}

 \renewcommand{\phi}{\varphi}
 \newcommand{\eps}{\varepsilon} 
 
 \newcommand{\IN}{\mathbb{N}}                  
                   
 \newcommand{\IQ}{\mathbb{Q}}                  
 \newcommand{\IR}{\mathbb{R}}

\newcommand{\CC}{\mathcal{C}}

 \theoremstyle{plain} 
 \newtheorem{Theorem}{Theorem}[section]
 \newtheorem{Lemma}[Theorem]{Lemma}
 \newtheorem{Proposition}[Theorem]{Proposition}
 \newtheorem{Corollary}[Theorem]{Corollary}
 \theoremstyle{definition} 
 \newtheorem{Definition}[Theorem]{Definition}
 \newtheorem{Remark}[Theorem]{Remark}
 \newtheorem{Example}[Theorem]{Example}
\newtheorem{Question}[Theorem]{Question}

\begin{document}

\title{W-measurable sensitivity of semigroup actions}

\author[F. Bozgan]{Francisc Bozgan}
\address{(Francisc Bozgan) Department of Mathematics, University of Chicago, Chicago, IL 60637, USA}
\email{fbozgan@math.uchicago.edu}

\author[A. Sanchez]{Anthony Sanchez}
\address{(Anthony Sanchez) Department of Mathematics, University of Washington, Box 354350, Seattle, WA, 98195-4530}
\email{asanch33@uw.edu}

\author[C.E. Silva]{Cesar E. Silva}
\address{(Cesar E. Silva) Department of Mathematics, Williams College, Williamstown, MA 01267, USA}
\email{csilva@williams.edu}

\author[J. Spielberg]{Jack Spielberg}
\address{(Jack Spielberg) School of Mathematical and Statistical Sciences, Arizona State University, Tempe, AZ 85287-1804, USA}
\email{jack.spielberg@asu.edu}

\author[D. Stevens]{David Stevens}
\address{(David Stevens)  Williams College, Williamstown, MA 01267, USA}
\email{davidfstevens14@gmail.com}

\author[J. Wang]{Jane Wang}
\address{(Jane Wang) Department of Mathematics, Massachusetts Institute of Technology, Cambridge, MA, 02138}
\email{janeyw@mit.edu}



\begin{abstract}
This paper studies the notion of W-measurable sensitivity in the context of countable semigroup actions. W-measurable sensitivity is a measurable generalization of sensitive dependence on initial conditions. In 2012, Grigoriev et. ~al.~  proved a classification result of conservative ergodic nonsingular dynamical systems that states all are either W-measurably sensitive or act by isometries with respect to some metric and have refined structure.  We generalize this result to a class of semigroup actions.  Furthermore, a counterexample is provided showing that W-measurable sensitivity is not preserved under factors. We also consider the restriction of W-measurably sensitive semigroup actions to sub-semigroups and show that the restriction remains W-measurably sensitive when the sub-semigroup is large enough (e.g. when the sub-semigroups are syndetic or thick).
\end{abstract}

\maketitle


\section{Introduction} 
Sensitive dependence on initial conditions has been a central notion in topological dynamical systems since being introduced formally by Guckenheimer in \cite{MR553966}, although these ideas go back to the work of Lorenz in the 1960's. Guckenheimer defined a transformation $T$ on a metric space $(X, d)$ to have \textit{sensitive dependence on initial conditions} if there exists a $\delta > 0$ such that for every $\varepsilon > 0$ and all $x \in X$, there exists a $y \in B (x, \varepsilon)$ and $n \in \IN$ such that $d(T^n(x), T^n(y)) > \delta$. The transformation $T$ defines a semigroup action of $\IN$ on $(X,d)$. Several generalizations of sensitivity to the context of other semigroup actions on metric spaces have been studied, for example in \cite{MR2367162}, \cite{MR2644895}, and \cite{MR2885999}. A different route of generalization that has been considered have been measure-theoretic versions of sensitive dependence on measure spaces (i.e. $\IN$-actions). Examples include \cite{MR1880515}, \cite{MR2154050}, and\cite{MR2811160} in the finite measure-preserving case and \cite{MR2993055}, \cite{MR2901201}, \cite{MR3326024}, and \cite{MR2415039} in the infinite measure and nonsingular case. In an unpublished manuscript that is in progress, S. Heinecke, E. Wickstrom and the third author, consider the equivalence of the various notions of measurable sensitivity for actions of $\mathbb N^d$.

In this paper we begin bridging the gap between these two generalizations by considering a measurable version of sensitive dependence on initial conditions in the context of semigroup actions. In particular a classification theorem of a measurable version of sensitivity, called W-measurable sensitivity, was proven in \cite{MR2901201}. That theorem stated roughly that a conservative, ergodic, nonsingular dynamical system ($\IN$-action) is either W-measurably sensitive or isomorphic modulo a measure zero set to an isometry. We will prove a generalization of this theorem for semigroup actions. We will do this by considering a class of countable semigroups that can be equipped with a partial order which will give us a notion of escaping to infinity as well as allowing us to generalize key lemmas from \cite{MR2901201}. To the best of our knowledge, this paper is the first attempt of trying to understand a measurable version of sensitive dependence on initial conditions to a large class of semigroup actions.

In section \ref{Restriction of Semigroup actions and W-Measurable Sensitivity}, we will consider the restriction of W-measurably sensitive semigroup actions to sub-semigroups and ask when is the restriction W-measurably sensitive. We will prove two distinct classes of sub-semigroups preserve W-measurable sensitivity in the conservative ergodic case: syndetic sub-semigroups and thick sub-semigroups. A corollary that syndetic sub-semigroups preserve W-measurable sensitivity will be that powers of a classical dynamical system ($\IN$-action) exhibit W-measurable sensitivity. We remark that studying the restriction to sub-semigroups and asking if W-measurable sensitivity is preserved is one way of understanding the set of semigroup elements for which separation of points is occurring and thus, natural to study.

\subsection{Organization}

The organization of the paper is as follows. In section \ref{Preliminary Definitions} we define the class of semigroup actions we consider and provide some preliminary definitions of standard dynamical notions in this context. In section \ref{Dynamical consequences of mu-compatibility} we discuss metrics that are compatible with our measurable systems and prove dynamical consequences of the metrics we consider.  The results proven in this section are key tools to the proof of our main classification result.  In Section \ref{W-measurable sensitivity} we define W-measurable sensitivity of semigroup actions and in section \ref{The Classification Theorem} we use the tools developed in previous sections to prove our main classification result that roughly states that under certain conditions, a measurable semigroup action is either W-measurably sensitive or is acting by isometries. This is the main dichotomy that generalizes the results of \cite{MR2901201}. In section, \ref{Restriction of Semigroup actions and W-Measurable Sensitivity}, we derive some hypotheses under which the restriction of a W-measurably sensitive semigroup action to a sub-semigroup continues to be W-measurably sensitive. Finally, in section \ref{Closing} we outline some questions and directions for future research.

\section{Preliminary Definitions}\label{Preliminary Definitions}

Let $(X,\mu)$ be a standard non-atomic Lebesgue space, and let $G$ be a countable abelian semigroup.  We consider a \textit{nonsingular dynamical system}, denoted $(X,\mu,G,T)$, which consists of a homomorphism $T$ from $G$ to the semigroup of nonsingular endomorphisms of $(X,\mu)$.  Thus, for each $g \in G$, $T_g : X \to X$ is measurable, and for any measurable set $A \subseteq X$, $T_g^{-1}(A)$ is a null set if and only if $A$ is a null set.  Moreover, $T_g \circ T_h = T_{g + h}$ for any $g$, $h \in G$. Occasionally we may assume that the action is measure-preserving (i.e. for every measurable set $A\subseteq X$ and any $g\in G$, $\mu(T_g ^{-1} A) = \mu (A)$) or that the measure space is finite. A subset $A \subseteq X$ is \textit{positively invariant (under $T$)} if $T_g(A) \subseteq A$ for all $g \in G$  (equivalently, $A \subseteq T_g^{-1}(A)$ for all $g \in G$).  We say $A$ is \textit{negatively invariant (under $T$)} if $T_g^{-1}(A) \subseteq A$ for all $g \in G$.  $A$ is \textit{invariant (under $T$)} if it is both positively and negatively invariant (equivalently, $T_g^{-1}(A) = A$ for all $g \in G$). We say a nonsingular action is \textit{ergodic} if every invariant measurable set is null or co-null.

We assume throughout that $G$ is a countable abelian \emph{cancellative} semigroup with \emph{no inverses}. A semigroup $G$ satisfies cancellation if $g + h = g + k$ implies $h = k$, and has no inverses if $g + h = 0$ implies $g = h = 0$.  For example, if $G$ is a pointed sub-semigroup of an abelian group, then $G$ has these properties.  In this case,  $G$ is equipped with a partial order by setting $g \le h$ if $h = g + x$ for some $x \in G$. For example, $G=\IN^2$ with the lexicographic order. Moreover, using the partial order we have a natural sense of what it means for a sequence of semigroup elements to go to infinity. We say $(g_i)\subseteq  G$ \emph{goes to infinity} (in the partial order of $G$) if for every $h\in G$, there is an $i_0\in \IN$ with $h\le g_i$ for all $i\ge i_0$. We denote this by $g_i\to\infty$.  While the conditions on $G$ appear to exclude any actual group, we recall that an abelian cancellative semigroup is embeddable in a group (see \cite{MR0218472}). So we can define $g_i\to\infty$ for a group that arises as an abelian cancellative semigroup, by asking that the sequence lives in the semigroup and that $g_i\to\infty$ in the semigroup. For simplicity, we continue to focus on abelian cancellative semigroup actions with no inverses.

We will endow our measure spaces with $\mu$-\emph{compatible metrics} (i.e.~ metrics for which non-empty open balls have positive measure). We assume throughout that the metrics are Borel measurable and are bounded by 1 (which can always be achieved by replacing a metric $d$ with $\frac{d}{1+d}$). By \cite{MR2415039} the topology generated by a $\mu$-compatible metric is separable. Hence, open sets are measurable as they are a countable union of balls. Of particular interest will be the metric $d_G$ defined by
$$d_G(x,y):=\sup _{g\in G }d(T_g x, T_g y)=\sup _{g\ge 0 }d(T_g x, T_g y)$$
for $x,y\in X.$ Since we only consider metrics that are bounded above by 1, then $d_G$ is always finite. As we consider multiple metrics, we will use the notation $B ^d(x,\varepsilon)$ for an open ball about $x$ of radius $\varepsilon$ to highlight the dependence on $d$.

\begin{Remark}\label{rem_motivation}

There are various measurable and topological properties of the dynamics of a single transformation that have different generalizations to the action of a semigroup of transformations.  The point is that for the semigroup $\IN$, a subset of $\IN$ is \textit{cofinal} for the partial order on $\IN$ if and only if it is an \textit{infinite} subset of $\IN$.  These two notions are no longer the same for general semigroups.  We will use the prefix ``$G$-'' to indicate properties defined by the notion of cofinality.

As an example, recall that a point $x \in X$ is called \textit{transitive} if the orbit of $x$ under $G$ is dense in $X$: $ \overline{\{T_g x : g \in G \}} = X$.  In the case $G = \IN$, i.e. iterates of a single transformation, this is equivalent to requiring that for each $n \in \IN$, $\{T_k x : k \ge n \}$ is dense in $X$.  For more general semigroups such as $\IN^d$, the two notions are not necessarily the same.  

\end{Remark}

With this remark in mind we make the following definitions.

\begin{Definition}

Let $(X,\mu,G,T)$ be a nonsingular abelian dynamical system.

\begin{enumerate}

\item $T$ is \textit{$G$-conservative} if for every measurable set $A$ of positive measure and every $g \in G$, there is $h \in G$ with $h \ge g$ such that $T_h^{-1}(A) \cap A$ has positive measure.

\end{enumerate}

Let $d$ be a $\mu$-compatible metric on $X$.

\begin{enumerate}
\addtocounter{enumi}{1}

\item A point $x \in X$ is \textit{$G$-transitive} if for each $g \in G$, the set $\{T_h x : h \ge g \}$ is dense in $X$.  We say that $T$ is \textit{$G$-transitive} if there exists a $G$-transitive point in $X$.

\item $T$ is \textit{$G$-uniformly rigid} if there is a sequence $g_i \in G$ such that $g_i \to \infty$ (in the partial order of $G$), and such that $T_{g_i} x \to x$ uniformly on $X$. 
\end{enumerate}

\end{Definition}

We close this section by recalling some standard notions in the semigroup setting that will be useful in later sections.

\begin{Lemma} \label{lem_invariantset}

Let $(X,\mu,G,T)$ be a nonsingular G-conservative action.  Let $A$ be a negatively invariant measurable set of positive measure.  Then there is an invariant measurable set $B \subseteq A$ such that $A \setminus B$ is a null set.

\end{Lemma}

\begin{proof}
Let $A$ be negatively invariant with positive measure.  We will first show that for each $g \in G$, $A \setminus T_g^{-1}(A)$ is a null set.  To do this, we will show that for each $h \ge g$ we have
\[
T_h^{-1}\bigl(A \setminus T_g^{-1}(A)\bigr) \cap \bigl(A \setminus T_g^{-1}(A)\bigr) = \emptyset.
\]
Then by $G$-conservativity, it will follow that $A \setminus T_g^{-1}(A)$ is a null set.  Let $h = g + k$.  Then
\[
T_h^{-1}(A \setminus T_g^{-1}(A)) = T_g^{-1}\bigl[ T_k^{-1}(A \setminus T_g^{-1}(A)) \bigr]
\subseteq T_g^{-1}\bigl( T_k^{-1}(A) \bigr) \subseteq T_g^{-1}(A),
\]
which is disjoint from $A \setminus T_g^{-1}(A)$.

Now let $B = \bigcap_{g \in G} T_g^{-1}(A)$.  Then $A \setminus B = \bigcup_{g \in G} (A \setminus T_g^{-1}(A))$ is a null set.  We must show that $B$ is invariant.  Let $h \in G$.  We have
\[
T_h^{-1}(B) = \bigcap_{g \in G} T_{h + g}^{-1}(A) = \bigcap_{k \ge h} T_k^{-1}(A) \supseteq \bigcap_{k \in G} T_k^{-1}(A) = B.
\]
Conversely, let $x \in T_h^{-1}(B)$.  Then $T_h x \in B$. Then for all $g \in G$, $T_hx \in T_g^{-1}(A)$, hence $T_g T_h x \in A$, hence $T_g x \in T_h^{-1}(A) \subseteq A$.  Thus $x \in \bigcap_{g \in G} T_g^{-1}(A) = B$.
\end{proof}

\begin{Corollary}\label{lem_posinv}

Let $(X,\mu,G,T)$ be a nonsingular ergodic G-conservative action. If $A$ is a positively invariant set of positive measure, then $A$ is co-null.

\end{Corollary}

\begin{proof}
If $A$ is not co-null, then $A^c$ is negatively invariant of positive measure. By the previous lemma, there is an invariant set $B \subseteq A^c$ such that $A^c \setminus B$ is null.  Then $B$ has positive measure.  By ergodicity, $B$ is co-null, hence $A^c$ is co-null, contradicting the assumption that $A$ has positive measure.
\end{proof}

\section{Dynamical consequences of $\mu$-compatibility}\label{Dynamical consequences of mu-compatibility}
In this section we prove some lemmas concering semigroup actions on a metric space.  These results will be used in the context of $\mu$-compatible metrics and are important in our proof of the main classification result for W-measurably sensitive semigroup actions in section \ref{The Classification Theorem}. 

We shall call a metric \textit{1-Lipschitz} (with respect to an action $G$) if $d(T_g x,T_g y)\le d(x,y)$ for every $x,y\in X$ and every $g\in G$. For example, $d_G$ is always a 1-Lipschitz metric. Recall that an action on a metric space is called \textit{minimal} if the orbit of every point is dense.

  In the case we have a 1-Lipschitz metric and transitivity, then we have refined structure of the action.

\begin{Lemma} \label{lem_isometric}

Let $(X,d)$ be a metric space.  If the action of an abelian semigroup $G$ on $X$ is 1-Lipschitz and $G$-transitive, then it is $G$-uniformly rigid, minimal, and isometric.

\end{Lemma}

\begin{proof}

Let $x$ be a $G$-transitive point, which will be fixed throughout the proof.   We first prove $G$-uniform rigidity.  Let $\eps > 0$.  We must show that for each $g \in G$, there is $g' \ge g$ such that for all $y \in X$, $d(T_{g'}y,y) < \eps$.  Let $g \in G$.  Since $x$ is $G$-transitive, there is $g' \in G$ such that $g' \ge g$ and $d(T_{g'} x,x) < \eps/3$.  Let $y \in X$.  Since $x$ is a transitive point, there is $h \in G$ such that $d(y,T_h x) < \eps/3$.  Then since the action is 1-Lipschitz, we have that $d(T_h x, T_h T_{g'} x) \le d(x, T_{g'} x) < \eps/3$ and $d(T_h T_{g'} x, T_{g'} y) = d(T_{g'} T_h x, T_{g'} y) \le d(T_h x, y) < \eps/3$.  Therefore
\[
d(y, T_{g'} y) \le d(y,T_h x) + d(T_h x, T_h T_{g'} x) + d(T_h T_{g'} x, T_{g'} y) < \eps.
\]
Now fix a sequence $(h_i)$ in $G$ such that $h_i \to \infty$, and for all $y \in X$, $\lim_i d(T_{h_i}y, y) = 0$.  Next, we prove that the action is minimal.  Let $y$, $z \in X$, let $\eps > 0$, and let $g \in G$.  Since $x$ is $G$-transitive, there is $g' \ge g$ such that $d(T_{g'}x, y) < \eps/2$.  Similarly, there is $g'' \ge g'$ such that $d(T_{g''} x, z) < \eps/2$.  Let $g'' = g' + h$, where $h \in G$.  Since the action is 1-Lipschitz, we have
\begin{align*}
d(T_h y,z) &\le d(T_h y, T_{g''} x) + d(T_{g''} x, z) \\
&= d(T_h y, T_h T_{g'} x) + d(T_{g''} x,z)
\le d(y, T_{g'} x) + d(T_{g''} x, z)
< \eps.
\end{align*}

Finally, we prove that the action is by isometries.  Since the action is 1-Lipschitz, we know that for all $y$, $z \in X$, and all $g \in G$, we have $d(T_g y, T_g z) \le d(y,z)$.  Suppose that for some $y$, $z \in X$, and some $g \in G$, we have that $d(T_g y, T_g y) < d(y,z)$.  Since $h_i \to \infty$, there is $i_0$ such that $h_i \ge g$ for all $i \ge i_0$.  For such $i$, let $h_i = g + k_i$, where $k_i \in G$.  Then
\[
d(T_{h_i} y, T_{h_i} z) = d(T_{k_i} T_g y, T_{k_i} T_g z) \le d(T_g y, T_g z) < d(y,z).
\]
 Hence, letting $i \to \infty$, we obtain $d(y,z) < d(y,z)$, a contradiction.
\end{proof}

Now we generalize Theorem 4.3 of \cite{MR1849204} (see also Proposition 5.6 of \cite{MR2901201}).  We let $\CC_d(X)$ be the space of continuous maps from $X$ to $X$ with the metric $d(S_1,S_2) = \sup_{x \in X} d(S_1 x, S_2 x)$.  Let $\Lambda = \{ S \in \CC_d(X) : S T_g = T_g S \text{ for all } g \in G \}$.  For $x \in X$ we define the \textit{evaluation map at $x$}, ev$_x : \Lambda \to X$, by ev$_x(S) = Sx$. This proposition will be one of the main results of the classification  theorem as it will allow us to relate the original $G$-action on $X$ to the induced action of $G$ on the space $\Lambda$. Furthermore, it will give some structure results for this new action. 

\begin{Proposition} \label{prop_akinglasner} 
Let $(X,d)$ be a metric space and suppose there is an action of an abelian semigroup $G$ on $X$ that is 1-Lipschitz and $G$-transitive.

\begin{enumerate}

\item   If $x \in X$ is a $G$-transitive point, then ev$_x$ is an isometry.  Furthermore, we have $\Lambda = \overline{ \{T_g : g \in G \} }$.

\item Suppose that $d$ is a complete metric on $X$.  Then for every $G$-transitive point $x \in X$, ev$_x$ is invertible.

\item The set $\Lambda$ is an abelian group. In particular, $T_g$ is invertible for each $g \in G$.

\end{enumerate}

\end{Proposition}

\begin{proof}

Let $x \in X$ be a transitive point.  If $S$, $S' \in \Lambda$, then for all $g \in G$ we have
\[
d(S T_g x, S' T_g x) = d(T_g S x, T_g S' x) \le d(Sx, S'x),
\]
since $T$ is 1-Lipschitz.  Since $x$ is transitive, we have
\[
d(S,S') = \sup_{x \in X} d(S T_g x, S' T_g x) \le d(Sx,S'x) \le d(S,S').
\]
Therefore ev$_x$ is isometric.  Now choose a sequence $(g_i)$ in $G$ such that $T_{g_i} x \to Sx$.  Then $d(T_{g_i},S) = d(T_{g_i}x,Sx) \to 0$, so that $\Lambda \subseteq \overline{ \{T_g : g \in G \} }$.  Since $\Lambda$ is closed in $\CC_d(X)$, the first part of the proposition is proved.

Now suppose that $d$ is complete.  Then $\CC_d(X)$ is also complete.  Let $y \in X$, and choose $(g_i)$ in $G$ so that $T_{g_i} x \to y$.  Since $(T_{g_i}x)$ is Cauchy, and ev$_x$ is isometric, we know that $(T_{g_i})$ is Cauchy.  Therefore there is $S \in \CC_d(X)$ such that $T_{g_i} \to S$.  Then $y = \lim_i T_{g_i} x = Sx = \text{ev}_x(S)$.  Thus ev$_x$ is onto.

Finally, we assume that $x$ is a $G$-transitive point of $X$, and we show that $\Lambda$ is an abelian group. We will prove this by showing  that all elements have an inverse. Let $S \in \Lambda$. We claim that $Sx$ is a transitive point.  For this, let $y \in X$ and $\eps > 0$.  Choose $g \in G$ such that $d(T_gx, Sx) < \eps/2$.  Since $x$ is $G$-transitive, there is $h \in G$ with $h \ge g$ such that $d(T_h x, y) < \eps/2$.  Write $h = g + k$.  Then since $T$ is 1-Lipschitz, we have
\[
d(T_k Sx,y) \le d(T_k Sx, T_h x) + d(T_h x, y) \le d(Sx,T_gx) + d(T_h x,y) < \eps.
\]
Therefore $Sx$ is a transitive point.  It follows from the second part of the proposition that ev$_{Sx}$ is onto, so there is $S' \in \Lambda$ such that $S' Sx = x$.  Since ev$_x$ is bijective, we have $S' S = I = S S'$. Lastly, it is clear that $\Lambda$ is abelian as all of its elements are formed as limits of commuting maps. We leave the details to the reader.

\end{proof}

We will now show that when we are working with a $\mu$-compatible metric, we can deduce $G$-transitivity. Consequently, for semigroup actions on metric spaces with a $\mu$-compatible metric, we have refined structure of the action and structural information on the group of commuting maps $\Lambda$.

\begin{Lemma} \label{lem_conull}

Let $(X,\mu,G,T)$ be $G$-conservative and ergodic, and let $d$ be a $\mu$-compatible metric on $X$.  Then there is a co-null invariant set $B \subseteq X$ such that every point of $B$ is $G$-transitive.

\end{Lemma}

\begin{proof}

Since $\mu$ is nonatomic, $d$ has no isolated points.  Since $d$ is $\mu$-compatible, $d$ is a separable metric (Lemma 1.1 of \cite{MR2415039}).  Let $S$ be a countable dense subset of $X$.  For $z \in S$, $g \in G$, and $r \in \IQ_+$, let
\[
A_{z,g,r} = \bigcup_{h \ge g} T_h^{-1}(B^d(z,r)),
\]
where $B^d(z,r)$ is the ball of radius $r$ around the point $z$ with the metric $d$. We show that $A_{z,g,r}$ is negatively invariant for $T$.  Let $k \in G$.  Then
\[
T_k^{-1}(A_{z,g,r}) = \bigcup_{h \ge g} T_k^{-1} ( T_h^{-1}(B^d(z,r)) )
= \bigcup_{h \ge g} T_{h + k}^{-1}(B^d(z,r))
= \bigcup_{h \ge g + k} T_h^{-1}(B^d(z,r))
\subseteq A_{z,g,r}.
\]
By Lemma \ref{lem_invariantset}, there is a measurable invariant set $B_{z,g,r} \subseteq A_{z,g,r}$ with $A_{z,g,r} \setminus B_{z,g,r}$ co-null.  Since $(X,\mu,G,T)$ is ergodic, $B_{z,g,r}$ is either a null set or a co-null set, and hence $A_{z,g,r}$ is either null or co-null.  Since $A_{z,g,r}$ contains the set $T_g^{-1}(B^d(z,r))$ of positive measure, $A_{z,g,r}$, and hence $B_{z,g,r}$, is co-null.  Let $B = \bigcap_{z,g,r} B_{z,g,r}$.  Then $B$ is co-null and invariant.  We claim that every point of $B$ is $G$-transitive.  To see this, let $y \in B$, $x \in X$, $g \in G$, and $\eps > 0$.  Choose $z \in S$ with $d(z,x) < \eps/2$.  Pick $r \in \IQ_+$ with $r < \eps/2$.    Since $y \in A_{z,g,r}$, there is $h \in G$ such that $h \ge g$ and $y \in T_h^{-1}(B^d(z,r))$.  Then $T_h y \in B^d(z,r)$.  Then $d(T_h y,x) \le d(T_h y, z) + d(z,x) < r + \eps/2 < \eps$.
\end{proof}

\section{W-measurable sensitivity}\label{W-measurable sensitivity}

In this section, we will introduce the generalization of W-measurable sensitivity to the semigroup setting and prove a technical lemma that will aid us in generalizing the classification theorem.
\begin{Definition}

Let $(X,\mu,G,T)$ be a nonsingular dynamical system. If $d$ is a $\mu$-compatible metric, $T$ is called \textit{W-measurably sensitive with respect to $d$} if there is a constant $\delta > 0$ such that for all $x \in X$,
\[
\limsup_{g \to \infty} d(T_g x,T_g y) > \delta
\]
for almost all $y \in X$, where the limsup is taken as $g \to \infty$ with respect to the partial order defined before Remark \ref{rem_motivation}.  $T$ is called \textit{W-measurably sensitive} if it is W-measurably sensitive with respect to every $\mu$-compatible metric.

\end{Definition}

It will follow from the theorems below that such actions exists. In fact, they are very common.

\begin{Remark}
(1) We obtain an equivalent characterization of W-measurably sensitivity  if we replace ``for all $x\in X$" with ``for almost every $x\in X$". Furthermore, we can replace ``$\limsup_{g\to\infty}d(T_g x, T_g y)>\delta$" with ``there exists $g\in G/\{0\}$ such that $d(T_g x, T_g y)>\delta$". These claims follow as in \cite{MR2901201}, Proposition 4.2.

(2) Recalling that any abelian cancelative semigroup is embeddable in a group, we see that we can define W-measurable sensitivity for any group that arises from such an embedding by asking that the semigroup it arises from is W-measurably sensitive.
\end{Remark}

We prove one more technical lemma that allows us to conclude when $d_G$ is $\mu$-compatible before generalizing the dichotomy from \cite{MR2901201}. Recall that, 
$$d_G(x,y):=\sup_{g\in G}d(T_gx,T_gy).$$

\begin{Lemma}\label{lem_technical}
Let $(X,\mu,G,T)$ be a nonsingular ergodic G-conservative action and $d$ a $\mu$-compatible metric on $X$. If the system is not W-measurably sensitive with respect to $d$, then there exists a positively invariant measurable set $X_1$ of full measure such that $d_G$ is a $\mu$-compatible metric for $(X_1,\mu, G, T)$, where $\mu$ and $T$ are the restrictions to $X_1$.
\end{Lemma}

\begin{proof}
Let $D(x) = \max \{ \eps \ge 0 : \mu(B^{d_G}(x,\eps)) = 0 \}$.  We aim to show that this is the zero function a.e., because then we will have that $d_G$ is $\mu$-compatible on a subset of full measure. Let $B$ denote the set of points where $D$ is non-zero, that is, $B = \{ x \in X : D(x) > 0 \}$. We need to show that $B$ is a null set. To do this we will need to use that $D\ge D\circ T_h$ for any $h\in G$ which will follow since $d_G$ is 1-Lipschitz. Using this inequality, we then show that $B$ is negatively invariant. Finally, we will argue by contradiction that $B$ is a null set. 

We show that $D\ge D\circ T_h$. Indeed, for $\eps > 0$, $h \in G$, and any $x \in X$, we have
\[
D \circ T_h (x) \ge \eps
\iff \mu(B^{d_G}(T_h x,\eps)) = 0
\iff \mu(T_h^{-1}(B^{d_G}(T_h x,\eps))) = 0.
\]
Since $d_G$ is 1-Lipschitz,
\[
y \in B^{d_G}(x,\eps)
\iff d_G(y,x) < \eps
\implies d_G(T_h y, T_h x) < \eps
\iff y \in T_h^{-1}(B^{d_G}(T_h x,\eps)).
\]
Thus $B^{d_G}(x,\eps) \subseteq T_h^{-1}(B^{d_G}(T_hx,\eps))$, and hence $\mu(B^{d_G}(x,\eps)) \le \mu(T_h^{-1}(B^{d_G}(T_h x,\eps)))$.  Therefore, $D \circ T_h (x) \ge \eps $ implies $ D(x) \ge \eps$.   This is true for all $\eps > 0$, and hence $D(x) \ge D \circ T_h (x)$.  Thus for all $h \in G$ we have $D \ge D \circ T_h$.

We use this to show that $B = \{ x \in X : D(x) > 0 \}$ is negatively invariant. Indeed,
\[
x \in T_h^{-1} B
\iff T_h(x) \in B
\iff D(T_h(x)) > 0
\implies D(x) > 0
\iff x \in B,
\]
i.e. $T_h^{-1}B \subseteq B$. Since $h \in G$ was arbitrary we have that $B$ is negatively invariant.

We now show that $\mu(B) = 0$. Suppose to the contrary that $B$ has positive measure. Then there is $r > 0$ such that $\mu(\{ D > r \} ) > 0$.  We claim that $\{ D > r \}$ is negatively invariant.  To see this, let $h \in G$ and $x \in T_h^{-1} \{ D > r \}$. Then $D(T_h x) > r$.  Since $D > D \circ T_h$ we have $D(x) > r$, hence $x \in \{ D > r \}$.  This verifies the claim.  Now by Lemma 2.3 there is an invariant set $A \subseteq \{ D > r \}$ with $\mu(A) = \mu(\{ D > r \})$, hence $\mu(A) > 0$.  By Corollary 2.4 we have that $A$ is conull, and hence $\{ D > r \}$ is conull.   Thus for almost every $x$, $D(x) > r$.  Thus for almost every $x$ we have that $B^{d_G}(x,r)$ is null.  For $y \in X$, $y \in B^{d_G}(x,r)$ if and only if for each $g \in G$ we have $d(T_g y, T_g x) < r$.  Thus $\{ y \in X : (\forall g \in G) \, ( d(T_g y, T_g x) < r) \}$ is null.  Then $\{ y \in X : (\exists g \in G) \, (d(T_g y, T_g x) \ge r \}$ has full measure.  Therefore for almost all $x$ and $y$ there is $g \in G$ such that $d(T_g x, T_g y) \ge r$.  By Remark 4.2 this implies W-measurable sensitivity, a contradiction.  Therefore $B$ is null, and hence $D = 0$ almost everywhere.  This finishes the proof.
\end{proof}

\section{The Classification Theorem}\label{The Classification Theorem}

In this section we generalize the dichotomy given for $\IN$-actions, first formalized and proved in  \cite{MR2901201}. Roughly, it states that a conservative ergodic classical dynamical system is W-measurably sensitive or acts by isometries and has refined structure. The aim of this section is to generalize this to the context of semigroup actions.
 
\begin{Theorem} \label{classif}
Let $(X,\mu,G,T)$ be a nonsingular ergodic G-conservative action.  If $T$ is not W-measurably sensitive, then $T$ is isomorphic mod measure 0 to a minimal $G$-uniformly rigid action by invertible isometries on an abelian Polish group.
\end{Theorem}

\begin{proof} Suppose that $(X,\mu,G,T)$  is not W-measurably sensitive, then by Lemma \ref{lem_technical} there is a positively invariant full measure set $X_1$ for which $d_G$ is $\mu$-compatible. We can restrict $\mu$ and $T$ to $X_1$. 

Now we have a quintuple $(X_1, \mu , G, T , d_G)$ where $d_G$ is $\mu$-compatible. We can apply our lemmas from section 3 to obtain dynamical information. By Lemma \ref{lem_conull} we know this action is G-transitive. Since $d_G$ is 1-Lipschitz, we have the action is  $G$-uniformly rigid, minimal, and isometric. 

We can complete the metric space $(X_1, d_G)$ to obtain $(X', d')$. This will be a Polish space since $d_G$ is separable. We extend the measure $\mu$ to $\mu'$ by defining a subset $A$ of $X'$ to be measurable if $A\cap X_1$ is measurable and $\mu'(A):= \mu(A\cap X_1)$. Since the action is by isometries, then it is continuous and so we can uniquely extend $T_g$, for every $g\in G$, to $(X',d')$ such that it continues to be an isometry. 

Finally,  by Proposition \ref{prop_akinglasner}, for any transitive point $x\in X'$, we have  the evaluation map is an invertible isometry from $(\Lambda,d_\Lambda)\to (X',d')$ where $d_\Lambda(S_1,S_2)=\sup_{x\in X}d'(S_1x,S_2x)$ and that $T_g$ is invertible for every $g\in G$.  We pullback $\mu'$ to a measure $\eta$ on $\Lambda$ via an evaluation map. Let $\tau_g:\Lambda\to \Lambda$ denote group rotation on $\Lambda$ by $T _g$ for any $g\in G$ and $\tau:G\to \text{End}(\Lambda,\eta)$ be the map $g\mapsto \tau_g$. By construction, we see $(\Lambda, \eta, G, \tau)$ is then measurably isomorphic to  $(X,\mu,G,T)$.
\end{proof}

\begin{Remark}
(1) Notice that the measurable isomorphism produced above is to the group $\Lambda$ consisting of maps that commute with the action $G$. This group comes equipped naturally with the Haar measure. In general, the measure $\eta$ that we produce on $\Lambda$ need not be the Haar measure. This is because $\eta$ is (basically) the pullback measure of a nonsingular ergodic measure and hence, need not even be measure-preserving. However, we show in the next section that in the finite measure-preserving case, the measure $\eta$ is the Haar measure on $\Lambda$.

(2) W-measurable sensitivity of semigroups is invariant under measurable isomorphism. The proof  follows as in the case $G=\IN$ (see \cite{MR2901201}). A related question is whether W-measurable sensitivity is preserved under \emph{factors}. We say that $(Y, \nu, G, T^\prime)$ is a factor of $(X, \mu, G, T)$ if there is a measurable $\varphi : X \rightarrow Y$ with $\varphi_* \mu = \nu$ and $T^\prime_g( \varphi(x)) = \varphi(T_g(x))$ for almost every $x \in X$. We provide a counterexample showing that W-measurable sensitivity is not preserved under factors:

Let $\alpha\notin \IQ$ and consider the product map  $T(x,y)= (2x\mod1, y+\alpha \mod 1)$ and the product system $(S^1\times S^1, m_{S^1}\times m_{S^1}, \IN, T)$ where $m_{S^1}$ is Haar measure on the circle $S^1$. By our choice of $\alpha$ we have the product system is an ergodic finite measure-preserving system.

By basic entropy theory (see chapter four of \cite{MR648108}) we know that the product system has positive entropy since the entropy of a product system is the sum of the entropies of each part, and the times two map has positive entropy. On the other hand, Proposition 7.1 from \cite{MR3326024} states that any positive entropy ergodic finite measure-preserving system must be W-measurably sensitive. We conclude that $(S^1\times S^1, m_{S^1}\times m_{S^1}, \IN, T)$ is W-measurably sensitive. Since the circle rotation is a factor of it and the circle rotation is not W-measurably sensitive, W-measurable sensitivity is not invariant under factors.

\end{Remark}


\subsection{Compact $\mu$-Compatible metrics}

As in \cite{MR2901201}, we can prove a stronger classification theorem in the finite measure-preserving case. We use this stronger classification to show that in the definition of W-measurable sensitivity, we only need to consider $\mu$-compatible metrics that are compact (i.e. metrics $d$ for which the topology generated by $d$ is compact).
\begin{Theorem} \label{thm_finiteclassification}
Let $(X,\mu,G,T)$ be an ergodic finite measure-preserving system.  If $T$ is not W-measurably sensitive, then $T$ is isomorphic mod measure 0 to a minimal $G$-uniformly rigid action by invertible isometries on a compact abelian group equipped with the Haar measure.
\end{Theorem}
 
\begin{proof}
If  the system is not W-measurably sensitive, then by Theorem \ref{classif}, it is isomorphic to a minimal $G$-uniformly rigid action by isometries on a Polish group which we denote by $(X,\mu,G,T)$. Since it is complete by construction it suffices to show that $(X,d)$ is totally bounded where $d$ is the $\mu$-compatible metric constructed in Theorem \ref{classif}. 

Let $\varepsilon>0$ and $f(x)=\mu(B(x,\varepsilon/2))$. This function is positive for any $x\in X$ by $\mu$-compatibility. We will show that it is constant. Indeed, since $\mu$ is measure-preserving and $d$ is an isometry, we have for each $x\in X$
$$f(T_gx)=\mu(B(T_gx,\varepsilon/2)) = \mu(T^{-1} _gB(T_gx,\varepsilon/2)) = f(x)$$
for any $g\in G.$ Since $f$ is continuous and our system is $G$-transitive, then $f$ is constant.

Then there is a largest finite collection $\{x_1,\ldots, x_n\}$ so that $B(x_i,\varepsilon/2)$ are mutually disjoint. Otherwise, we would contradict that the measure space is finite. Moreover, since this collection is the largest possible, then for every $x\in X$, we have some $i$ so that $d(x,x_i)<\varepsilon$. Hence, $X=\bigcup_{i=1} ^n B(x_i,\varepsilon) $ and so $(X,d)$ is totally bounded.

Given that the measurable isomorphism is to a compact metrizable group, then the measure $\eta$ produced in Theorem \ref{classif} must be the Haar measure. Indeed, the action we produce of $G$ on $\Lambda$ is a minimal rotation and such actions have the Haar measure as the only invariant measure (see Theorem 6.20 of \cite{MR648108}).
\end{proof}

\begin{Corollary}  \label{thm_finiteclassification} A finite measure-preserving ergodic system  $(X, \mu, G,T)$ is
W-measurably sensitive with respect to all $\mu$-compatible metrics if and only if
it is W-measurably sensitive with respect to all compact $ \mu$-compatible metrics.

\end{Corollary}
\begin{proof} The forward direction is immediate.

To show the converse, suppose  $(X, \mu, G,T)$ is W-measurably sensitive with
respect to all $compact$ $\mu$-compatible metrics but there is some $\mu$-compatible
metric for which $T$ is not W-measurably sensitive. By Theorem \ref{thm_finiteclassification}, we
must then have that it is isomorphic to an isometry on a compact Abelian
group with respect to some metric that is $\mu$-compatible. 

By supposition, the isomorphic system is W-measurably sensitive with respect
to the compact metric and an isometry for this compact metric,
which is impossible.
\end{proof}

\section{Restriction of Semigroup actions and W-Measurable Sensitivity}\label{Restriction of Semigroup actions and W-Measurable Sensitivity}

The results in this section have to do with the relation between  W-measurable sensitivity of $(X,\mu,G,T)$ and W-measurable sensitivity of $(X,\mu, H,T)$  for sub-semigroups $H$ of $G$. We denote this restriction by $T|_H$. 

\subsection{Syndetic sub-semigroups}\label{Syndetic sub-semigroups}
The first result concerns what might be termed \textit{cofinite} sub-semigroups.  As a special case we will see that a power of a single transformation exhibiting W-measurable sensitivity also is W-measurably sensitive if it is still conservative and ergodic.

Denote for a semigroup element $f$ and subset $A$ of $G$, the backwards translation by $$f^{-1}A:=\{g\in G:f+g\in A\}.$$

\begin{Definition}

Let $G$ be an abelian semigroup.  A sub-semigroup $H$ of $G$ is called \textit{syndetic} if $G$ can be covered by finitely many backwards translations of $H$. That is, there is a finite set $F\subset G$ so that
$$\bigcup _{f\in F}f^{-1}H=G.$$ 
\end{Definition}

\begin{Remark} Recall that an abelian cancellative semigroup is embeddable in a group. Then syndeticity is equivalent to requiring that the corresponding group for $H$ is of finite index in the corresponding group for $G$.
\end{Remark}

\begin{Example}For $G=\IN$ the subset $H=k\IN$  is syndetic for any $k\in \IN$.
\end{Example}

\begin{Theorem} \label{syn}
Let $(X,\mu,G,T)$ be a nonsingular $G$-conservative action that is W-measurably sensitive. If $H$ is a sub-semigroup of $G$ which is syndetic in $G$ and $T|_H$ is $H$-conservative and ergodic, then $T|_H$ is a W-measurably sensitive action.
\end{Theorem}

\begin{proof}

Suppose that $T|_H$ is not W-measurably sensitive. By the classification theorem, Theorem \ref{classif}, there is some metric $d$ for which  $T|_H$ acts by invertible isometries with respect to $d$. Let $\delta$ be the sensitivity constant for $T$, relative to the metric $d$. Finally choose a finite subset $F\subset G$ as in the definition of syndeticity.

By W-measurable sensitivity of the $G$-action we have for every $x\in X$ and almost every $y\in X$, there is $g\in G$ with $d(T_gx,T_gy)>\delta.$  Also, since $H$ is syndetic, then there is $f'\in F$ so that $g+f'\in H$. Applying once again the definition of syndeticity to the element $f'$ shows that there is $f\in F$ with $f+f'\in H$. Now, since $H$ acts by isometries and $f+f', g+f'\in H$, we conclude
$$d(T_gx,T_gy) = d(T_f T_{f'} T_gx,T_f T_{f'} T_gy) = d(T_fx,T_f y).$$
Hence, by W-measurable sensitivity and $H$ being syndetic, we have that for every $x\in X$ and almost every $y\in X$ there is an $f=f(g)\in F$ where $d(T_fx,T_fy)>\delta.$

Now, let us cover $X$ by a countable, measurable, pairwise disjoint cover $(B_i)_{i\in\IN}$ where $\text{diam}(B_i)<\delta/2$ for every $i.$ For every $f \in F$, we can define the function $\varphi_f : X \rightarrow \IN$ that sends a point $x \in X$ to the index of the element of the cover $\{B_i\}$ that the point $T_f x$ is in. We can also define $\varphi : X \rightarrow \IN^F$ by $\varphi(x) = (\varphi_f(x))_{f \in F}$. 

Since $\IN^F$ is countable, there exists an $m \in \IN^F$ such that $\varphi^{-1}(m)$ has positive measure. Let $x\in \varphi^{-1}(m)$. Then, for almost every $y$ we know by W-measurable sensitivity that 
$d(T_{f}x,T_{f}y)>\delta$ which implies that $\varphi_{f}(x)$ is not equal to $\varphi_{f}(y)$. Hence, $y\notin \varphi^{-1}(m)$. Therefore, $\varphi^{-1}(m)$ has $0$ measure, which is a contradiction. 
\end{proof}

\begin{Example}

\item Suppose $(X,\mu,\IN,T)$ is an ergodic measure-preserving system on a probability space. If it exhibits W-measurable sensitivity, then for any $k \in \IN$, if $T^k$ is ergodic then $T^k$ also exhibits W-measurable sensitivity.   To see this take $H=k\IN$ and $F=\{0,1,...,k-1\}$. In particular, any totally ergodic measure-preserving system on a probability space that exhibits W-measurable sensitivity has that all of its powers are W-measurably sensitive.

Suppose $(X,\mu,\IN^d,T)$ exhibits W-measurable sensitivity. Let ($v_1$, $\ldots$, $v_d) \in \IN^d$ be linearly independent, and let $H = \IN$-span$\{v_1, \ldots, v_d \}$.  If $T|_H$ is ergodic then $T|_H$ also exhibits W-measurable sensitivity.

\end{Example}

\subsection{Thick sub-semigroups}\label{thick sub-semigroups}
 In this section we consider another class of sub-semigroups that preserve W-measurable sensitivity under a restriction.

\begin{Definition}

 A sub-semigroup $H$ of $G$ is called \textit{thick in $G$} if for all finite sets $F\subset G$, there is a $p\in G$ such that $F+p :=\{f+p:f\in F\}\subset H$.

\end{Definition}

\begin{Example} Let $H$ be a \emph{cone} in $\IN^2$. That is, the set of elements in between two rays anchored at the origin in $\IN^2$. Then $H$ is a thick sub-semigroup in $\IN^2$. This generalizes to cones in $\IN^d$ by replacing two rays anchored at the origin to $d$-rays. See Figure  \ref{fig:cones}.

\begin{center}

\begin{tikzpicture}

\node[red][above] at (-2,.3) {$H$};
\node[red] at (-3, 0) {\textbullet};
\draw[black] (-3,0) [->] to (-3,3);
\draw[red] (-3,0) [->] to (0,3);
\draw[red] (-3,0) [->] to (0,1);
\draw[black] (-3,0) [->] to (0,0);

\end{tikzpicture}

\captionof{figure}{A cone in $\IN^2$.} \label{fig:cones}

\end{center}

\end{Example}

\begin{Remark}
Note that $H=k\IN$ is syndetic, but not thick. On the other hand, a cone in $\IN^2$ with rays at angles 0 and $\pi/4$ is thick, but not syndetic. Hence, the two notions are distinct and in order to show that the restriction of a  W-measurably sensitive action onto such subsets remains W-measurably sensitive may require different techniques.
\end{Remark}

\begin{Theorem}\label{thick}
Let $(X,\mu,G,T)$ be a nonsingular $G$-conservative action that is W-measurably sensitive. If $H$ is a sub-semigroup of $G$ which is thick in $G$ and $T|_H$ is $H$-conservative and ergodic, then $T|_H$ is a W-measurably sensitive action.
\end{Theorem}

\begin{proof} 
Suppose that $T|_H$ is not W-measurably sensitive. By the classification theorem, Theorem \ref{classif}, there is some metric $d$ for which  $T|_H$ acts by invertible isometries with respect to $d$. Let $\delta$ be the sensitivity constant for $T$, relative to the metric $d$.

 Let $x\in X$ and $\varepsilon = \delta.$ Since $T$ is W-measurably sensitive, then there is $y\in X$ and $g\in G$ with the property that $d(x,y)<\varepsilon$ and $d(T_gx,T_gy)>\delta$. 

Since $H$ is thick in $G$ then there is some $p\in G$ with $g+p\in H$. Using once more that $H$ is thick in $G$ on the set $\{g+p,p\}$ we get that there is $q$ with $g+p+q=:h$ and $p+q$ in $H$. Finally, we use that $H$ acts by isometries to deduce that
\begin{align*}
\delta&<d(T_gx,T_gy)\\
&= d(T_{p+q}T_gx,T_{p+q}T_gy)\\
&= d(T_hx,T_hy)\\
&=d(x,y)<\varepsilon = \delta,
\end{align*}
a contradiction.
\end{proof}

\section{Closing Remarks and Further Questions}\label{Closing}

We close the paper with a few remarks. There have been some recent results concerning {\it sensitive dependence on initial conditions} which is the topological notion that motivated W-measurable sensitivity.  In particular, there have been results about sensitive dependence for semigroup actions such as \cite{MR2367162} and \cite{MR2885999}. The results in these papers are purely topological while we focus on the measurable aspects of sensitivity. Moreover, the semigroup actions in these papers are by a topological semigroup that is also a {\it C-semigroup}. A semigroup $G$ is called a C-semigroup if for every $g\in G$, the closure of $G\setminus\{hg|h\in G\}$ is compact in $G$. We make no topological assumptions on our semigroups. Furthermore, the condition of being a C-semigroup excludes semigroups such as $\IN^d$ or $\IQ^d _{\ge0}$ for $d\ge2$, which our results allow. However, the family of C-semigroups includes one-parameter semigroups, which our results do not. We thus have the following question:

\begin{Question}

To what class of semigroups can we extend W-measurable sensitivity and the classification theorem, Theorem \ref{classif}?

\end{Question}

Restriction of a W-measurably sensitive semigroup action is not automatically W-measurably sensitive since separation of points only needs to occur at a single group element by Remark 4.2.  The last section gave some examples where the restriction of a W-measurably sensitive action to sub-semigroups continued to be W-measurably sensitive provided the sub-semigroup was ``large". There are other notions of ``largeness" for semigroups such as being piecewise syndetic, central, an IP set, etc. We refer the reader to \cite{Berg} for definitions.   It would be interesting to understand which of these notions of ``largeness" would preserve W-measurable sensitivity under restriction. For example, a generalization of Theorem \ref{syn} could be:

\begin{Question}

Let $(X,\mu,G,T)$ be a nonsingular G-conservative action that is W-measurably sensitive. If $H$ is a sub-semigroup of $G$ which is piecewise syndetic in $G$  and $T|_H$ is $H$-conservative and ergodic, then is $T|_H$ a W-measurably sensitive action?

\end{Question}

\subsection{Acknowledgments}This research was initiated by the Ergodic Theory group of the the 2013 SMALL Undergraduate Research program at Williams College with support provided by the National Science Foundation REU Grant DMS-0850577 and the Science Center of Williams College. Both JW and AS were supported by the National Science Foundation Graduate Research Fellowship under Grants No. 1256082 and  No. 1122374, respectively. Parts of this research also formed part of the senior research thesis of AS while a student at Arizona State University. 

Since August 2019, CS has been serving as a Program Director in the Division of Mathematical Sciences at the National Science Foundation (NSF), USA, and as a component of this job, he received support from NSF for research, which included work on this paper. Any opinions, findings, and conclusions or recommendations expressed in this material are those of the authors and do not necessarily reflect the views of the National Science Foundation.

We are indebted to the referee for a careful reading of the manuscript and several comments and suggestions that improved our paper. 
\bibliographystyle{plain}

\end{document}